\documentclass[reqno]{amsart}
\numberwithin{equation}{section}
\usepackage{luatex85}
\usepackage{bbm}
\usepackage[T1]{fontenc}
\usepackage{Preamble}
\pdfoutput=1

\title{New examples of non-unique enhancements for triangulated categories}

\author{A. Rizzardo}
\email[Alice Rizzardo]{alice.rizzardo@liverpool.ac.uk}
\address{Department of Mathematical Sciences\\
University of Liverpool\\
Mathematical Sciences Building\\
Liverpool L69 7ZL\\
United Kingdom}

\author{J. Symons}
\email[Julie Symons]{julie.symons@uantwerpen.be}
\address{Department of Mathematics \\ University of Antwerp \\ Campus Middelheim \\ Middelheimlaan 1 \\ 2020 Antwerp \\ Belgium}

\author{M. Van den Bergh}
\email[Michel Van den Bergh]{michel.vandenbergh@uhasselt.be}
\address{Universiteit Hasselt\\Martelarenlaan 42\\3500 Hasselt\\Belgium}

\thanks{
  The first author is a lecturer at the University of Liverpool. The second author holds a PhD fellowship from the Research Foundation - Flanders (FWO).
  The third author is a senior researcher at the Research Foundation Flanders
(FWO). This project has received funding from the European Research Council (ERC) under the
European Union's Horizon 2020 research and innovation programme (grant agreement No 885203)}

\makeatletter
\@namedef{subjclassname@2020}{
  \textup{2020} Mathematics Subject Classification}
\makeatother


\let\marginparold\marginpar
\def\marginpar#1{\marginparold{\raggedright \tiny #1}}
\begin{document}
\counterwithout{thm}{subsection} 
\counterwithin{thm}{section}
\setcounter{tocdepth}{2}

% ----------------------------------------------------------------------------------------
% ABSTRACT
\begin{abstract}
  We present a general procedure for constructing triangulated categories, linear over a field, with distinct enhancements. Some of our examples can be equipped with a (non-degenerate) $t$-structure, thereby showing that the existence of a $t$-structure does not imply uniqueness of enhancements, whether in the strong or weak sense (depending on the example).
\end{abstract}

\maketitle
\tableofcontents

% ----------------------------------------------------------------------------------------
% INTRODUCTION
\settocdepth{section}
\section{Introduction} 

Triangulated categories appear naturally in many areas of mathematics. However, they are often not sufficiently well-behaved to support many constructions one would like to perform, such as forming functor categories, studying deformations, carrying out base changes, or constructing gluings. A standard way to overcome these shortcomings is to work with an \emph{enhancement}: a richer structure, such as a dg- or $A_{\infty}$-category, or a stable $\infty$-category, whose homotopy category recovers the given triangulated category.

\ver

A fundamental question is whether such enhancements exist and, if so, whether they are (strongly) unique (see \S\ref{sec:enhancements} for precise definitions). This question has attracted considerable attention over the past two decades. The survey \cite{CSTour} gives a snapshot of the state of the art at the time, but substantial progress has been made since. Many triangulated categories arising in practice are now known to admit unique enhancements. For instance, Canonaco, Neeman and Stellari proved in \cite{CNSuniqueness} that the derived category of any abelian category admits a unique enhancement. They further showed in \cite{uniquenessenhancementsdual} that the bounded and bounded below derived categories of modules over the dual numbers have strongly unique enhancements, and the same holds for derived categories of hereditary categories.

On the other hand, uniqueness does not hold in complete generality. The first and third author constructed triangulated categories with multiple enhancements, even over a field \cite{RVdBUnic}, and further examples were obtained in \cite{infiniteenhancements}. These examples, however, are derived categories of graded fields and are therefore particularly simple. In particular, they do not admit a non-degenerate $t$-structure.

\ver

This raises a natural question. A recurring theme in recent years has been that a triangulated category equipped with a non-degenerate (or even bounded) $t$-structure often retains substantially more information than the underlying triangulated category alone. This phenomenon appears in reconstruction results, for instance in Neeman's theory of approximable triangulated categories \cite{approximability} and his work on $K$-theory \cite{Neemanheart}, as well as in deformation theory, where a $t$-structure can recover information that is invisible at the purely triangulated level. One may therefore wonder whether the presence of a bounded, or more generally non-degenerate, $t$-structure similarly constrains possible enhancements, forcing them to be unique.

\ver 

The main result of this paper shows that this is not the case. We develop a general framework for constructing triangulated categories with multiple enhancements, inspired in part by the ideas of \cite{RVdBUnic}, resulting in much less trivial triangulated categories. This procedure yields large families of examples, including triangulated categories equipped with a bounded $t$-structure that admit multiple enhancements. Thus, the existence of a non-degenerate $t$-structure does not imply uniqueness of enhancements, neither in the strong nor in the weak sense.

\ver

\zcref{thm - undirected} in this paper is similar to Corollary 8.3 in \cite{patchkoria}. The results in \cite{patchkoria} are obtained using $\infty$-categorical techniques and Postnikov truncations (i.e. homotopical information above a certain degree is discarded). The paper \cite{patchkoria} also contains a number of examples of non-strongly-equivalent triangulated categories. 

In contrast, our proof relies on $A_\infty$-techniques, which allow us to formulate and establish \zcref{thm - undirected} in a more concise and direct way, cutting the needed work substantially. It should be remarked that the bounds of \cite{patchkoria} are sharper. We expect to be able to match them by refining our $A_{N}$-approach, but pursuing this was not one of the aims of the present paper. 

The concrete examples we produce are particularly well-behaved and work over any field of characteristic zero. As already mentioned, the examples in our paper admit a $t$-structure (this is not discussed in \cite{patchkoria}). Moreover, we go further: we show that our examples admit non-equivalent enhancements rather than just non-strongly equivalent ones.

\ver

To state our main results, we first need a definition. A (strictly unital) $k$-linear $A_\infty$-category is called \emph{$N$-formal} (for $N\ge 2$) if it 
has a minimal model with
 $m_3,\ldots,m_{N}=0$. 

\begin{thm - undirected}
    Let $\mc{A}$ be an $N$-formal $k$-linear $A_{\infty}$-category for $N \geq 3$. Assume furthermore that $H^\ast(\A)$ is right graded coherent
and that it has right finite  global dimension~$q$. Then there is an equivalence of triangulated categories
\begin{equation}   \tag{\ref{eq:body}}
\D_{\tn{fp}}(\mc{A}) \cong \D_{\tn{fp}}(H^{*}(\mc{A})) \quad \tn{ if } N \geq 8q+7,
\end{equation}
    where $\D_{\tn{fp}}(-)$ denotes the full triangulated subcategory of  the derived category $\D(-)$ (see below)
spanned by the objects~$X$ for which $H^{*}(X)$ is graded finitely presented over $H^\ast(\A)$; which moreover extends the identity isomorphism between the full subcategories of $\D_{\tn{fp}}(\A)$ and $\D_{\tn{fp}}(H^\ast(A))$ spanned by the shifts of the objects in $\A$ and $H^\ast(\A)$.
\end{thm - undirected}

This theorem is proved using the fact that $\A$ has \emph{perfect dimension} $q$. We introduce this notion in \zcref{section - perfect dimension}, together with a class of examples. Having perfect dimension~$q$ implies that $\D_{\tn{fp}}(\A)$ is equivalent to the homotopy category of the idempotent completion of the $q$-truncated twisted complexes, that is, to $H^{0}(\Split_{\infty}(\Tw_{\leq q} \A))$ (see \zcref{sec:idempotent_completion}), and similarly for $\D_{\tn{fp}}(H^{*}(\A))$. Combined with the fact that the triangulated structure of a pretriangulated $A_{\infty}$-category is determined by the $A_{\infty}$-operations $m_{1},m_{2}$ and $m_{3}$, via Massey products (see \zcref{appendix}), this yields \zcref{thm - undirected}.

\ver

By tweaking the hypotheses of \zcref{thm - undirected}, one may obtain better bounds in \zcref{eq:body}. For instance, when $H^\ast(\A)$ is `directed' in a suitable sense, one obtains the following result.
\begin{thm - directed}
  Let $\mc{A}$ be an $N$-formal $k$-linear $A_{\infty}$-category for $N \geq 3$ such that $H^\ast(\A)$
is directed  of length $l$ (see \zcref{def:directed} below). Then \zcref{thm - undirected} holds with the bound \zcref{eq:body} replaced by $N\ge l+3$.
\end{thm - directed}
If $\A$ in
\zcref{thm - undirected} or \zcref{thm - directed} is 
non-formal, it follows readily that the triangulated category
$\D_{\tn{fp}}(\A)$ admits two non-strongly equivalent (cf. \ref{sec:enhancements}) enhancements, namely
$\Split_{\infty}(\Tw_{\leq q} \A)$ and
$\Split_{\infty}(\Tw_{\leq q} H^{*}(\A))$. This result can be improved to non-equivalent  
enhancements in some cases, as we observe below.

\ver

In \zcref{sec - examples}, we provide two different methods for constructing $N$-formal $A_{\infty}$-categories satisfying the conditions of these theorems.
They both depend on the choice of a suitable field $K\supset k$.
\begin{enumerate}
  \item
In \zcref{subsec - ex1}, we let $\mf{g}_1$ be the arrow category consisting of two copies of $K$
glued via a $k$-$k$-bimodule $M$, which is a non-trivial extension of $K$ by $K[N]$, where $N \geq 4$. This example is directed.
\item In \zcref{subsec - ex2},  
  we let $\mf{g}_2$ be the graded polynomial algebra in one variable $K[t]$, where $t$ is of degree $1-N$ and $N \geq 7$ is odd\footnote{The attentive reader
    will notice that the bound $N\geq 7$ is better than the bound $N\geq 8q+7=15$ in \zcref{thm - undirected}. This is because we actually use the variant
\zcref{thm - undirected improved}, which has the bound $4q+3$.}, and endow it with a minimal $A_{\infty}$-structure $(0,m_{2},0,\ldots,0,m_{N+1},m_{N+2},\ldots)$ using a non-trivial Hochschild cocycle in $\tn{HH}^{N}_{k}(K)$. 
\end{enumerate}
With some additional arguments, we show that the two enhancements in these examples are not equivalent. By considering the standard $t$-structure and observing that the examples are cohomologically concentrated in nonpositive degrees, this illustrates that enhancements of a triangulated category with a non-degenerate $t$-structure need not be unique.

\section{Acknowledgements}
This material is partially based upon work supported by the National Science Foundation under Grant No. DMS-1928930 and by the Alfred P. Sloan Foundation under grant G-2021-16778, while the authors were in residence at the Simons Laufer Mathematical Sciences Institute (formerly MSRI) in Berkeley, California, during the Spring 2024 semester. 

The authors thank Gustavo Jasso for pointing out the paper \cite{patchkoria}. The first author would like to thank Amnon Neeman for many conversations on triangulated categories and their enhancements, and the University of Hasselt for its hospitality during a research visit.

% ----------------------------------------------------------------------------------------
% CONVENTIONS
\section{Notation and conventions}
\label{sec:notation_and_conventions}
We fix a field $k$ of arbitrary characteristic. In \zcref{sec - examples}, we let $k$ be of characteristic zero. All categories and constructions are $k$-linear.  Unless otherwise specified modules are right modules.

\ver

\subsection{Gradings} \label{sec:gradings} A \emph{grading} will always be a $\mathbb{Z}$-grading. Elements of graded objects are always assumed to be homogeneous and the same is true for morphisms between graded objects. For a graded object $A^{\bullet} = (A_{i})_{i \in \mb{Z}}$ in an abelian category $\mathbb{A}$ we put $\Sigma A^{\bullet} := (A_{i+1})_{i \in \mb{Z}}$. The degree of a graded morphism $f:A^{\bullet} \rightarrow B^{\bullet}$ is denoted by $|f|$. We define $\Sigma f:\Sigma A^{\bullet} \rightarrow \Sigma B^{\bullet}$ via
$(\Sigma f)_i=(-1)^{|f|}f_{i+1}$. If the degree of a morphism is not explicitly specified then it is assumed to be zero. By $s^n_A:A^\bullet\rightarrow \Sigma^n A^\bullet$ (or simply $s^n$) we denote the morphism of degree $-n$  which is the identity on every component. Note that $s_{\Sigma A}=-\Sigma s_A$ and for $n>0$: $s^n_A=s_{\Sigma^n A}\cdots s_{\Sigma A}s_A$. Following \cite{Lefevre}, we also put $\omega_A^n:=s_A^{-n}:\Sigma^n A\rightarrow A$ (and we usually write $\omega^n:=\omega_{A}^n$).

The morphisms $s$ and $\omega$ are \emph{degree change} operators. If $V^{\bullet}$ 
is a graded object in $\tn{Vect}(k)$ and $x \in V$, then we denote its degree by $|x|$.
The element $sx\in \Sigma^\bullet V$ is the same as  $x$ but with degree $|sx|=|x|-1$.

\ver

\subsection{Triangles and triangulated categories} \label{sec:enrichment}
A \emph{triangle} in a category $\B$ enriched in graded vector spaces (such as $H^\ast(\A)$ for an $A_\infty$-category $\A$) is a sequence of
morphisms
\[
  A\rightarrow B\rightarrow C \xrightarrow{(1)} A.
\]
If $\B$ is equipped with a shift functor $\Sigma$ (e.g.\ if $\B$ is triangulated), then it is naturally equipped with
a graded enrichment
\[
\operatorname{Hom}^\ast(A,B)=(\operatorname{Hom}(A,\Sigma^n B))_n.
\]
In that case a triangle will also be written as
\[
  A\rightarrow B\rightarrow C \rightarrow \Sigma A. 
\]
\subsection{Dg-categories} We identify dg-algebras with dg-categories with a single object. Given two dg-categories $\A$ and $\B$, we denote the dg-category of dg-$\A$-$B$-bimodules by $\tn{Mod}(\A,\B)$ and its derived category by $\mc{D}(\A,\B)$.

\ver

\subsection{\texorpdfstring{$A_{\infty}$}{Ainfty}-categories} \label{subsec - Ainfty}
Likewise we identify $A_\infty$-algebras with $A_\infty$-categories with a single object. All $A_{\infty}$-categories are assumed to be strictly unital and we follow the notation and terminology of \cite{Lefevre}:

Given an $A_{\infty}$-category $\A$, we denote the dg-category of strictly unital right (so contravariant) $A_{\infty}-\A$-modules with strictly unital $A_\infty$-morphisms by $\mc{C}_{\infty}(\A)$ and we put $\tn{Mod}_{\infty}(\A) := Z^{0}(\mc{C}_{\infty}(\A))$. There is the Yoneda $A_{\infty}$-embedding $$\mc{Y}: \A \to \mc{C}_{\infty}(\A): X \mapsto \A(-,X).$$
We will often omit $\mc{Y}$ and identify $\A$ with its image.
 We write $\tn{Mod}_{\infty}(\A)^{\tn{strict}}$ for the subcategory of $\tn{Mod}_{\infty}(\A)$ with only the strict morphisms. There are (triangulated) equivalences by \cite[5.2.0.2 and 4.1.3.14]{Lefevre}:
 \[
   H^{0}(\mc{C}_{\infty}(\A)) \cong \tn{Mod}_{\infty}(\A)[\tn{Qis}^{-1}] \cong \tn{Mod}_{\infty}(\A)^{\tn{strict}}[\tn{Qis}^{-1}],
 \] where $\tn{Qis}$ denotes the class of quasi-isomorphisms. 
We will refer to any of these models as the \ti{derived category} of $\A$ and denote it  by $\mc{D}(\A)$.

Similarly, given another $A_{\infty}$-category $\B$, we use the notations $\C_\infty(\A, \B)$, $\tn{Mod}_{\infty}(\A,\B)$ and $\tn{Mod}_{\infty}(\A,\B)^{\tn{strict}}$,
$\mc{D}(\A,\B)$ for the analogously defined categories of $A_{\infty}-\A-\B$-bimodules.

By \cite[4.2.0.4 and 4.1.3.14]{Lefevre},
\begin{equation} \label{eq - Ainf dg bimod}
    \mc{D}(U(\A),U(\B)) \cong \mc{D}(\A,\B),
\end{equation}
where $U(\A)$ is the dg-category analogue to the dg-algebra $A'$ showing up in the proof of loc.cit. In fact, one can adapt the proof to use the right adjoint $U(-)$ of \cite[Proposition 2.1]{COS}, which associates a universal dg-category to a strictly unital $A_{\infty}$-category. If $\A$ is moreover a dg-category, then there is a map (the counit)
\begin{equation} \label{eq - dg qiso U}
    U(\A) \xrightarrow{\sim} \A
\end{equation}
that is a quasi-isomorphism, again by \cite[Proposition 2.1]{COS}.

\subsection{Enhancements}
\label{sec:enhancements} 
We follow \cite[Definition 2.1]{LuntsOrlov} except that we
place ourselves in the $A_\infty$-context throughout. This
is justified by the results in \cite{COS}. For an introduction
to the required $A_\infty$-notions, see \zcref{sec:preliminaries}.

An $A_\infty$-enhancement (or
simply enhancement) of a triangulated category $\mc{T}$ is a pair
$(\mc{C},E)$, where $\mc{C}$ is a pretriangulated $A_\infty$-category and
$E: H^{0}(\mc{C}) \to \mc{T}$ is an equivalence of triangulated
categories. Two enhancements $(\mc{C},E)$ and $(\mc{C}',E')$ are said to be \emph{equivalent}
if there is  an $A_\infty$-functor
$F: \mc{C} \to \mc{C}'$ such that $H^{0}(F)$ is an 
equivalence. If, in addition, we can choose $F$ in such a way
that there exists an isomorphism of exact functors
$E \cong E' \circ H^{0}(F)$, then we say that $(\mc{C},E)$ and $(\mc{C}',E')$
are \emph{strongly equivalent}.

As in \cite[Definition 2.2]{LuntsOrlov}, an enhancement $(\mc{C},E)$ is \emph{(strongly) unique} if it is (strongly) equivalent to any other enhancement.

We note that all triangulated categories that occur ``in nature'' come with a standard enhancement. E.g. if $\A$ is an $A_\infty$-category, then $\C_\infty(\A)$
provides a standard enhancement for $\D(\A)$. If $\A$ is a dg-category, then we can consider both a standard dg-enhancement of the derived category, like $\hproj(\A)$ or $\tn{semi-free}(\A)$, and $\C_\infty(\A)$. The two are strongly equivalent by \cite[Bottom of page 93]{Positselski}.

% ----------------------------------------------------------------------------------------
% 1 - PRELIMINARIES
\section{Preliminaries}
\label{sec:preliminaries}
\setcounter{subsection}{1}
\subsection{\texorpdfstring{$A_{\infty}$}{Ainfty}-categories}

We follow \cite{Lefevre} and \cite{RVdB}.

\begin{definition}
    An \textit{$A_{\infty}$-category} $\A$ is the data of 
    \begin{itemize}
        \item[-] a set of objects $\tn{Ob}(\A)$;
        \item[-] for every couple $(A,A')$ of objects of $\A$, a graded $k$-vector space $\A(A,A')$, called \ti{the Hom-space between $A$ and $A'$}. We call a homogeneous element of $\A(A,A')$ a \ti{morphism};
        \item[-] for every $i\geq 1$ and every sequence $(A_{0},\ldots,A_{i})$ of objects in $\A$, \ti{higher compositions} $$b_{i}: \Sigma \A(A_{i-1},A_{i}) \otimes \cdots \otimes \Sigma \A(A_{0},A_{1}) \to \Sigma \A(A_{0},A_{i})$$ of degree 1 satisfying the equation
        \begin{equation} \label{eq - Ainfty b} \sum_{j+k+l=i} b_{j+1+l}(1^{\otimes j} \otimes b_{k} \otimes 1^{\otimes l}) = 0;\end{equation}
        \item For each object $A$ an \ti{identity or unit morphism} $\tn{id}_{A}\in \A(A,A)_{0}$ satisfying
        \begin{align*}
            b_{2}(s \tn{id}_{A}, sg) &= sg &&\tn{if $n \geq 2$,}\\
            b_{2}(sf, s\tn{id}_{A}) &= (-1)^{\vert f \vert}sf &&\tn{if $n \geq 2$,}\\
            b_{i}(\ldots,s\tn{id}_{A},\ldots) &= 0 &&\tn{if $n = 1$ or $3 \leq i \leq n$},
        \end{align*}
        where $f$ and $g$ are composable morphisms.
    \end{itemize}
\end{definition}

\begin{remark} \label{remark - An-cat small values} The higher compositions can also be expressed as operations
$$m_{i}:\A(A_{i-1},A_{i}) \otimes \cdots\otimes \A(A_{0},A_{1}) \to \A(A_{0},A_{i})$$
of degree $2-i$, via the relation
$b_{i} := -sm_{i}\omega^{\otimes i}$. Using the Koszul sign rule, we also obtain the
converse formula: $m_{i} = -(-1)^{\frac{i(i-1)}{2}} \omega b_{i} s^{\otimes i}$. Equation \zcref{eq - Ainfty b} then corresponds to 
\begin{equation} \label{eq - Ainfty m}
    \sum_{j+k+l=i} (-1)^{j+kl} m_{j+1+l}(1^{\otimes j} \otimes m_{k} \otimes 1^{\otimes l}) = 0,
\end{equation}
which is relation $(*)_{i}$ in \cite[Définition 1.2.1.1]{Lefevre}. The operation $m_{1}$ is sometimes called the \ti{differential}, and denoted by $d$, and $m_{2}$ is also called the \ti{composition}, where we write $m_{2}(g,f) = gf$. For small values of $n$, we obtain the following relations (where extra signs appear because of the Koszul sign convention):
\begin{align} \label{eq - b2 to m2}
    b_{1}(sf) &= -sm_{1}(f), \nonumber \\
    b_{2}(sg,sf) &= (-1)^{\vert g \vert} sm_{2}(g,f), \\
    b_{3}(sh,sg,sf) &= -(-1)^{\vert g \vert}sm_{3}(h,g,f). \nonumber
\end{align}
\end{remark}

\subsection{\texorpdfstring{$A_{m}$}{Am}-functors}
\begin{definition}
    An $A_{m}$-functor $f: \A \to \B$ between two $A_{\infty}$-categories $\A$ and $\B$ with $m \leq n$ is the data of
    \begin{itemize}
        \item[-] a map on objects $f: \tn{Ob}(\A) \to \tn{Ob}(\B)$;
        \item[-] for every $1 \leq i \leq m$ and every sequence $(A_{0},\ldots,A_{i})$ of objects in $\A$, compositions $$f_{i}: \Sigma \A(A_{i-1},A_{i}) \otimes \cdots \otimes \Sigma \A(A_{0},A_{1}) \to  \Sigma \B(f(A_{0}),f(A_{i}))$$ of degree zero satisfying the equation
        \begin{equation} \label{eq - Ainfty f}
            \sum_{j+k+l=i} f_{j+1+l}(1^{\otimes j} \otimes b_{k} \otimes 1^{\otimes l}) = \sum_{\substack{i_{1} + \cdots + i_{r} = i}} b_{r}(f_{i_{1}}\otimes \cdots \otimes f_{i_{r}}),
        \end{equation}
        and for each object $A$,
         \begin{align*}
            f_{1}(s\tn{id}_{A}) &= s \tn{id}_{f(A)} && \tn{ if }m \geq 1, \\
            f_{i}(\ldots,s\tn{id}_{A},\ldots) &= 0 &&\tn{ if } m \geq 2.
         \end{align*}
    \end{itemize}
   An $A_\infty$-functor is the data $(f,(f_i)_i)$ such that $(f,(f_i)_{i\le m})$ defines an $A_m$-functor for each $m$.
\end{definition}

\begin{remark}
    The compositions $f_{i}$ can also be expressed as operations
    $$g_{i}: \A(A_{i-1},A_{i}) \otimes \cdots \otimes \A(A_{0},A_{1}) \to \B(A_{0},A_{i})$$
    of degree $1-i$ via the relation $f_{i} := sg_{i}\omega^{\otimes i}$. Condition \zcref{eq - Ainfty f} then translates to $(**)_{i}$ in \cite[Définition 1.2.1.2]{Lefevre}.
\end{remark}

\begin{definition}
\label{def:functors}
    Let $f: \A \to \B$ be an $A_{m}$-functor between $A_{\infty}$-categories. We say that
    \begin{itemize}
        \item[-] $f$ is \ti{strict} if $m \geq 1$ and $f_{i}=0$ for $i\geq 2$. Equivalently, $f_{1}$ commutes with higher compositions with arity at most $m$.
        \item[-] $f$ is \ti{fully faithful} if it is strict and for all $A,A' \in \tn{Ob}(\A)$, $f_{1}$ induces an isomorphism of vector spaces $\A(A,A') \to \B(f(A),f(A'))$.
    \end{itemize}
    Assume that $m \geq 2$. Then
    \begin{itemize}
        \item[-] $f$ is \ti{quasi-fully faithful} if $H^{*}(f)$ is fully faithful.
        \item[-] $f$ is \ti{quasi-essentially surjective} if $H^{*}(f)$ is essentially surjective.
        \item[-] $f$ is a \ti{quasi-isomorphism} if $H^{*}(f)$ is an isomorphism.
        \item[-] $f$ is a \ti{quasi-equivalence} if $H^{*}(f)$ is an equivalence.
    \end{itemize}
\end{definition}

\subsection{Reminder about twisted complexes}
\label{sec:twisted}

Below, $\A$ is a fixed $A_\infty$-category.
\begin{definition}
    $\tn{Free}(\A)$ is the graded graph with as objects finite (possibly empty) formal direct sums of formal shifts of objects of $\A$,
        $$A = \oplus_{i \in I} \Sigma^{a_{i}} A_{i}$$
        where $A_{i} \in \tn{Ob}(\A), a_{i} \in \mb{Z}$ and $\vert I \vert < \infty$. Morphisms are defined by
        $$\tn{Free}(\A)(\oplus_{i \in I} \Sigma^{a_{i}} A_{i}, \oplus_{j \in J} \Sigma^{b_{j}}  B_{j}) = \bigoplus_{\substack{i \in I, j \in J}} \Sigma^{b_{j}-a_{i}}\A(A_{i},B_{j}).$$

        If we consider an element $f \in \A(A,B)$ as an element of $\tn{Free}(\A)(\Sigma^{a}A,\Sigma^{b}B)$, then we will denote it by $\sigma^{b-a}f$ with $\vert \sigma^{b-a} f \vert = \vert f \vert - (b-a)$.  Formally, $\sigma^{b-a}$ is the graded morphism which is the composition
\[
\A(A,B) \xrightarrow{s^{b-a}} \Sigma^{b-a}\A(A,B) = \tn{Free}(\A)(\Sigma^{a}A,\Sigma^{b}B)
\]
where $s^{b-a}$ was introduced in \zcref{sec:gradings}.
\end{definition}

\begin{remark}
Note that the notation $\sigma^{b-a}f$ is only defined for $f \in \A(A,B)$ and cannot be extended to all of $\tn{Free}(\A)$. For example, one cannot make sense of $\sigma  \sigma^{-1}f$ as  $\sigma$ and $\sigma^{-1}$ would have to anti-commute.
\end{remark}

It is easy to see that
    $\tn{Free}(\A)$ is an $A_\infty$-category with higher compositions for $1 \leq i \leq n$ defined as
    $$b_{\tn{Free}(\A),i}(s \sigma^{a_{i} - a_{i-1}}f_{i}, \ldots, s \sigma^{a_{1}-a_{0}}f_{1}) = \pm \sigma^{a_{i}-a_{0}}b_{i}(sf_{i},\ldots,sf_{1})$$
        on a sequence
        $$\Sigma^{a_{0}}A_{0} \xrightarrow{\sigma^{a_{1}-a_{0}}f_{1}} \Sigma^{a_{1}} A_{1} \xrightarrow{\sigma^{a_{2} - a_{1}}f_{2}} \cdots \xrightarrow{\sigma^{a_{i} - a_{i-1}}f_{i}} \Sigma^{a_{i}}A_{i},$$
        where the sign is determined by the usual Koszul sign convention (with the rule $s\sigma = - \sigma s)$. $\tn{Free}(\A)$ is equipped with a strict 
(see \zcref{def:functors})
$A_{\infty}$-endofunctor $\Sigma$ given by 
        \begin{align*}
            \Sigma(\oplus_{i} \Sigma^{a_{i}} A_{i}) &= \oplus_{i} \Sigma^{a_{i}+1} A_{i}, \\
            \Sigma(\sigma^{a} f) &= (-1)^{a} \sigma^{a}f.
        \end{align*}
\begin{definition} A \emph{twisted complex} over $\A$ is a pair $(A,\delta_A)$, where $A\in \Ob(\tn{Free}(\A))$ and $\delta$ is a strictly lower triangular
  matrix in $\tn{Free}(\A)(A,A)$ (with respect to the decomposition of $A$ as a formal
  sum of shifted objects in $\A$) that is in addition 
  a \emph{Maurer-Cartan} element,
 i.e. it satisfies
\begin{equation}
\label{eq:maurercartan}
\sum b_{i}(s \delta_{A},\ldots,s\delta_{A}) = 0.
\end{equation}
Here and below, we note that sums as in \zcref{eq:maurercartan} are finite because of the assumption that $\delta_A$ is strictly lower triangular.
  \end{definition}
  Twisted complexes form an $A_\infty$-category $\Tw \A$ with
  \[
    (\Tw \A)((A,\delta_A), (B,\delta_B))=\tn{Free}(\A)(A,B),
    \]
    and higher compositions 
    for $i\ge 1$ defined by
        \begin{multline*}
            b_{\Tw \A,i}(sg_{i},\ldots,sg_{1}) \\
            = \sum_{\substack{l_{0},\ldots,l_{i} \\ h = i + \sum_{j} l_{j}}} b_{\tn{Free}(\A),h}(\underbrace{s\delta_{i},\ldots,s\delta_{i}}_{l_{i}},sg_{i},\underbrace{s\delta_{i-1},\ldots,s\delta_{i-1}}_{l_{i-1}},\ldots,\underbrace{s\delta_{1},\ldots,s\delta_{1}}_{l_{1}},sg_{1},\underbrace{s\delta_{0},\ldots,s\delta_{0}}_{l_{0}})
        \end{multline*}
        on a sequence
        $$(A_{0},\delta_{0}) \xrightarrow{g_{1}} (A_{1},\delta_{1}) \xrightarrow{g_{2}} \cdots \xrightarrow{g_{i}} (A_{i},\delta_{i}).$$

The category $H^0(\Tw \A)$ has a natural triangulated structure. The shift is inherited from $\tn{Free}(\A)$ and  the \emph{cone} of a morphism $f: (A,\delta_{A}) \to (B,\delta_{B})$ in $Z^{0}(\Tw \A)$ is the object $(\Sigma A \oplus B, \delta_{C(f)}) \in \Tw \A$ such that
\begin{equation}
  \label{eq:cone}
\delta_{C(f)} = \begin{pmatrix} -\delta_{A} & 0 \\ \sigma^{-1}f & \delta_{B}
                      \end{pmatrix}.
\end{equation}
This yields a so-called \emph{standard distinguish triangle} in $\Tw \A$,
    \begin{equation}
\label{eq:nonstandard}
\begin{tikzcd}
        A & B & {(C(f),\delta_{C(f)})} & A,
        \arrow["f", from=1-1, to=1-2]
        \arrow["i", from=1-2, to=1-3]
        \arrow["p", from=1-3, to=1-4]
        \arrow["{(1)}"', from=1-3, to=1-4]
    \end{tikzcd}
\end{equation}
    where
    $$i = \begin{pmatrix} 0 \\ \tn{id}_{B} \end{pmatrix}, \qquad p = \begin{pmatrix} \sigma^{-1}\tn{id}_{A} & 0 \end{pmatrix}.$$
    The distinguished triangles in $H^0(\Tw \A)$ are the triangles that are isomorphic in $H^0(\Tw \A)$ to a standard distinguished triangle. It will also be convenient to make the following definition:
    \begin{definition} \label{def - cone} A triangle (see \zcref{sec:enrichment}) in $H^\ast(\A)$
      is distinguished if this is the case for its image in $H^\ast(\Tw \A)$.
      \end{definition}

Following \cite{BondalKapranov}, we say that $\A$ is \emph{pretriangulated} if the inclusion $\A\rightarrow \Tw \A$ is a quasi-equivalence (see \zcref{def:functors}).
In that case, the distinguished triangles and the shift inherited from $H^0(\Tw \A)$ make $H^0(\A)$  into a triangulated category.

If $f: \A \to \B$ is an $A_{\infty}$-functor between $A_{\infty}$-categories, then $f$ induces
an $A_\infty$-functor $\Tw f:\Tw \A \to \Tw \B$, defined on objects $(A,\delta_{A}) \in \Tw \A$ as
$$(\Tw f)(A,\delta_{A}) = (f(A),s^{-1}\sum_{i \leq q} f(s\delta_{A},\ldots,s\delta_{A})),$$  and on a sequence of composable arrows
$$(A_{0},\delta_{0}) \xrightarrow{g_{1}} (A_{1},\delta_{1}) \xrightarrow{g_{2}} \cdots \xrightarrow{g_{d}} (A_{d},\delta_{d})$$
as
\begin{equation}
  \label{eq:tw_functor}
  (\Tw f)_d(sg_{1},\ldots,sg_{d}) = \sum f_{d+i_{0}+\cdots+i_{d}}((s\delta_{d})^{\otimes i_{d}},sg_{d},(s\delta_{d-1})^{\otimes i_{d-1}},\ldots,sg_{1},(s\delta_{0})^{\otimes i_{0}}).
  \end{equation}

\subsection{Idempotent completion}
\label{sec:idempotent_completion}
If $\A$ is an additive category, then its \emph{idempotent completion} is written as $\Split(\A)$ (cf. \cite{BalmerSchlichting}). Recall the following
\begin{thm}[{\cite[Theorem 1.12]{BalmerSchlichting}}] \label{thm:id_triangulated}
If $\A$ is triangulated, then
$\Split(\A)$ is also triangulated with the shift functor inherited from $\A$ and with
as distinguished triangles the direct summands of distinguished triangles in $\A$. 
\end{thm}

If $\A$ is an $A_\infty$-category, then it has a natural enhancement $\Split_{\infty}(\A)$ in the sense that
\[
  H^{0}(\Split_{\infty}(\A)) \cong \Split(H^{0}(\A));
\]
see \cite[(4c) before Lemma 4.7]{Seidel}. We call $\Split_{\infty}(\A)$ the \emph{$A_\infty$-idempotent completion} of $\A$.
\zcref{thm:id_triangulated} has an  $A_\infty$-version.
\begin{thm}[{\cite[Lemma 4.8]{Seidel}}] If $\A$ is pretriangulated then so is $\Split_\infty(\A)$.
\end{thm}
\subsection{Truncated twisted complexes}
The formula \zcref{eq:tw_functor} cannot be used directly if $f$ is an $A_m$-functor since the expansion on the right of \zcref{eq:tw_functor} may involve expressions
with more than $m$ arguments. We can solve this by restricting the Maurer Cartan elements.

\begin{definition}[\ti{Truncated twisted complexes in $\A$.}] \label{truncated}
  Given $0 \leq q \leq n$, $\Tw_{\leq q} \A$ is the $A_\infty$-subcategory of $\Tw \A$ with
  objects $(A,\delta)$, where $\delta$ is a block matrix with $\le q+1$ blocks. We call the objects of $\Tw_{\le q}\A$ \ti{truncated twisted complexes of length (at most) $q$ in $\A$.}
  \end{definition}
  
\begin{proposition}[{\cite[Lemma 6.7]{RVdB}}] \label{prop - functoriality *q and Twq}
  Let $f: \A \to \B$ be an $A_{m}$-functor between $A_{\infty}$-categories, where $q \leq m$. Then $f$ induces
  an $A_{\lfloor \frac{m-q}{q+1} \rfloor}$-functor $\Tw_{\leq q}f:\Tw_{\leq q}\A \to \Tw_{\leq q}\B$ via the formula 
  \zcref{eq:tw_functor}. Moreover $f\mapsto \Tw_{\le q} f$ is strictly compatible with composition.
\end{proposition}

% ----------------------------------------------------------------------------------------
% 3 - A TRIANGULATED EQUIVALENCE FOR N-FORMAL AINFTY-CATS
\settocdepth{subsection}
\section{Perfect dimension for \texorpdfstring{$A_\infty$}{Ainfty}-categories} \label{section - perfect dimension}
\subsection{Perfect dimension}
It will be convenient to make the following definition.
\begin{definition} 
 \label{def:perfect_dimension} 
An $A_\infty$-category $\A$ has \emph{perfect dimension} $\le q$ if the inclusion $$\langle \A \rangle_{q+1} \hookrightarrow \Perf(\A)$$ is an equivalence, where
\begin{itemize}
\item $\langle \A\rangle_{q+1}$ (see \cite[\S 2.2]{BondalVdB}) is the full subcategory of $\D(\A)$ spanned by objects that can be created from objects in $\A$ via sums, summands, and at most $q$ cones;
\item $\Perf(\A)$ (see \cite[\S 2.2]{BondalVdB}) is the category of perfect $\A$-modules.
\end{itemize}
\end{definition}
\begin{remark}
Note that the perfect dimension is not a Morita invariant of $\A$ since it depends on the choice of the specific generator $\A$ of  $\Perf(\A)$.
\end{remark}

\begin{lemma} \label{lem:perfect_dimension}
If $\A$ has perfect dimension $\le q$, then the \emph{totalization functor}, i.e. the strict $A_\infty$-functor that collapses a twisted complex into a single $A_\infty$-$\A$-module:
\[
\operatorname{Tot}:\Tw_{\le q} \A \rightarrow \C_\infty(\A)
\]
induces an equivalence
\begin{equation}
\label{eq:spliteq}
H^0(\operatorname{Tot}):\Split(H^0(\Tw_{\le q} \A)) \rightarrow \Perf(\A)
\end{equation}
    \begin{proof}
        By \cite{BondalVdB}, $\tn{Perf}(\A) = \bigcup_q \langle \A \rangle$, and by definition (cf. \cite[\S 2.2]{BondalVdB}), $$\langle \A \rangle_{q+1} = \tn{smd}(\tn{smd}(\tn{Free}(\A)) \star \cdots\star \tn{smd}(\tn{Free}(\A)))$$
        with $q+1$ factors, where
        \begin{itemize}
            \item[-] $\tn{smd}(\tn{Free}(\A))$ denotes the minimal strictly full subcategory of $\D(\A)$ containing $\tn{Free}(\A)$ that is closed under (possible) direct summands;
            \item[-] $\tn{Free}(\A)  \star \tn{Free}(\A)$ is the strictly full subcategory of $\D(\A)$ whose objects $X$ occur in a triangle $A \to X \to A' \rightarrow \Sigma A$ for $A,A' \in \tn{Free}(\A)$.
        \end{itemize} 
        We can then apply \cite[Lemma 2.2.1]{BondalVdB} to obtain the equality
        $$\langle \A \rangle_{q+1} = \tn{smd}(\tn{Free}(\A)  \star \cdots \star \tn{Free}(\A)).$$
        Since $\tn{Free}(\A)  \star \cdots \star \tn{Free}(\A) \hookrightarrow H^{0}(\tn{Tot})(H^{0}(\Tw_{\leq q} \A))$ and $\tn{Perf}(\A)$ is thick, we have the inclusions
        $$\langle \A \rangle_{q+1} \hookrightarrow \Split(H^{0}(\tn{Tot})(H^{0}(\Tw_{\leq q} \A))) \hookrightarrow \tn{Perf}(\A).$$
        The claim then follows from the fact that $\A$ has perfect dimension $q$.
    \end{proof}
\end{lemma}
If follows from \zcref{eq:spliteq} that $\Split(H^0(\Tw_{\le q} \A))$ is a triangulated category (or equivalently, that $\Split_\infty(\Tw_{\le q} \A)$ is pretriangulated).
For the future, we record the following observation.
\begin{lemma}
\label{lem:splitlemma}
A triangle in $\Split(H^0(\Tw_{\le q} \A))$ is distinguished if and only if it is isomorphic to a direct summand of a distinguished triangle (as defined in \zcref{def - cone}) in
$H^0(\Tw_{\le 1}(\Tw_{\le q}\A))$.
\end{lemma}
\begin{proof}
According to \cite[Lemma 1.6]{BalmerSchlichting}, a direct summand of a distinguished triangle is distinguished. This proves one direction. For the other direction, note that it follows from the construction in the proof of \cite[Lemma 1.13]{BalmerSchlichting} that any distinguished triangle in $\Split(H^0(\Tw_{\le q} \A)) $
is a direct summand of a standard distinguished triangle (see \zcref{eq:nonstandard}) in $\Tw_{\le 1}(\Tw_{\le q}\A)$. This yields what we want.
\end{proof}

\subsection{Homologically coherent \texorpdfstring{$A_\infty$}{Ainfty}-categories}
Below, we identify a particular class of $A_\infty$-categories with finite perfect dimension.
Recall that a category is \emph{coherent} if its category of finitely
presented right modules is abelian. For a category enriched in graded
abelian groups, one defines \emph{graded coherent} in a similar way.
\begin{definition}
We say that an $A_\infty$-category $\A$ is \emph{homologically coherent} if $H^\ast(\A)$ is graded right coherent. If this is the case, then we define 
$\D_{\tn{fp}}(\A)$ as the full subcategory of $\D(\A)$ spanned by the objects~$X$ for which $H^{*}(X)$ is graded finitely presented over $H^\ast(\A)$.
\end{definition}
\begin{lemma} If $\A$ is homologically coherent, then $\D_{\tn{fp}}(\A)$ is a thick triangulated subcategory of $\D(\A)$.
\end{lemma}
\begin{proof} 
It is clear that $\D_{\tn{fp}}(\A)$  is thick. To prove that it is triangulated, 
we have to show that $\D_{\tn{fp}}(\A)$ is closed under cones. If $Z=\operatorname{cone}(X\rightarrow Y)$ with $X,Y\in \D_{\tn{fp}}(\A)$, then we have a long exact sequence of right $H^\ast(\A)$-modules,
\begin{equation}
\label{eq:fp}
H^\ast(X)\rightarrow H^\ast(Y) \rightarrow H^\ast(Z)\rightarrow \Sigma H^\ast(X) \rightarrow \Sigma H^\ast(Y).
\end{equation}
Since $H^\ast(\A)$ is graded coherent, the finitely presented graded modules form an abelian category. Moreover finitely presented graded modules are closed under extensions. Hence,
\zcref{eq:fp} implies that $H^\ast(Z)$ is coherent.
\end{proof}
\begin{lemma}
Assume that $\A$ is homologically coherent. Then we have the following sequence of categories and functors:
\begin{equation}
\label{eq:bigdiagram}
\begin{tikzcd}
\D_{\tn{fp},\pdim\le q}(\A)\arrow[hookrightarrow]{r}& \langle \A\rangle_{q+1} \arrow[hookrightarrow]{r}&\Perf(\A) \arrow[hookrightarrow]{r}&\D_{\tn{fp}}(\A),
\end{tikzcd}
\end{equation}
where, in addition to the notation introduced in \zcref{def:perfect_dimension} and \zcref{lem:perfect_dimension}, we have that
 $\D_{\tn{fp},\pdim\le q}(\A)$ is the full subcategory of $\D_{\tn{fp}}(\A)$ spanned by objects $X$ such that $H^\ast(X)$ has graded projective dimension $\le q$
over $H^\ast(\A)$.
\end{lemma}
\begin{proof} The inclusion $\Perf(\A)\hookrightarrow\D_{\tn{fp}}(\A)$ follows from the fact that $\Perf(\A)$ is classically generated by $\A$ (see \cite{BondalVdB}),
$\A\subset \D_{\tn{fp}}(\A)$ and $\D_{\tn{fp}}(\A)$ is thick.

The only other non-standard fact is the inclusion $\D_{\tn{fp},\pdim\le q}(\A)\hookrightarrow  \langle \A\rangle_{q+1}$. We prove this by induction.
To avoid confusion, we will write the Yoneda functor explicitly. 

Note that $H^\ast(\Y(U))=H^\ast(\A)(-,U)$ is a graded projective
 $H^\ast(\A)$-module and that every (finitely generated) projective $H^\ast(\A)$-module is a summand of a (finite) sum of $\Sigma^i H^\ast(\Y(U))$.

First, we consider the case $q=0$. Let $X\in \D(\A)$ be such that $H^\ast(X)$ is a finitely generated projective $H^\ast(\A)$-module.
Then there exist objects
$X_1,\ldots,X_n$ in $\A$ and shifts $j_1,\ldots, j_n$ such that we have maps of graded $H^\ast(\A)$-modules,
\begin{equation}
\label{eq:decomp}
\bigoplus_i \Sigma^{j_i}H^\ast(\Y(X_i))\overset{\alpha}{\twoheadrightarrow} H^\ast(X) \hookrightarrow \bigoplus_i \Sigma^{j_i}H^\ast(\Y(X_i)),
\end{equation}
such that the composition is an idempotent $e$ in
\[
\End_{H^\ast(\A)}(\bigoplus_i \Sigma^{j_i}H^\ast(\Y(X_i))=H^0(\bigoplus_{i,k}
{\tn{Free}(\A)}(\Sigma^{j_i}X_i,\Sigma^{j_k}X_k))=\End_{\D(\A)}(\bigoplus_i \Sigma^{j_i}\Y(X_i)).
\]
Likewise, the map $\alpha$ in \zcref{eq:decomp} corresponds to an element of 
\[
\operatorname{Hom}_{\D(\A)}(\bigoplus_i \Sigma^{j_i}\Y(X_i),X).
\]
Let $Z\in \D(\A)$ be the image of $e$ (recall that $D(\A)$ is idempotent complete by e.g. \cite[Lemma 4.19]{Rouquier}). Applying $H^\ast(-)$ to 
\[
Z\hookrightarrow \bigoplus_i \Sigma^{j_i}\Y(X_i) \xrightarrow{\alpha} X
\]
shows that the composition $Z\rightarrow X$ is a quasi-isomorphism. Hence
$X\cong Z\in \langle \A\rangle_1$.

\medskip

Assume now that $q>0$ and let $X\in \D_{\tn{fp}, \pdim\le q}(\A)$. Then  there exist objects
$X_1,\ldots,X_n$ in $\A$ and shifts $i_1,\ldots,i_n$ such that we have a surjective map of graded $H^\ast(\A)$-modules
\[
\bigoplus_i \Sigma^{j_i}H^\ast(\Y(X_i))\overset{\alpha}{\twoheadrightarrow} H^\ast(X).
\]
As above, $\alpha$ corresponds to a map $\bigoplus_i \Sigma^{j_i}\Y(X_i) \xrightarrow{\alpha} X$. Consider the corresponding distinguished triangle:
\begin{equation}
\label{eq:induction}
Y\rightarrow \bigoplus_i \Sigma^{j_i}\Y(X_i) \xrightarrow{\alpha} X \rightarrow \Sigma Y.
\end{equation}
Applying $H^\ast(-)$ yields a short exact sequence:
\[
0\rightarrow H^\ast(Y)\rightarrow \bigoplus_i \Sigma^{j_i}H^\ast(\Y(X_i)) \xrightarrow{\alpha} H^\ast(X)\rightarrow 0.
\]
We obtain that $\operatorname{pdim} H^\ast(Y)\le q-1$ and, by induction, we may assume that $Y\in \langle \A\rangle_q$. Since $X=\operatorname{cone}(Y\rightarrow \bigoplus_i \Y(X_i))$, we conclude that $X\in \langle \A\rangle_{q+1}$.
\end{proof}
We now obtain the following corollary to this lemma.

\begin{proposition}  \label{prop - Tot qequiv} Let $\A$ be a homologically coherent $A_\infty$-category
such that $H^\ast(\A)$ has finite global dimension~$q$.
Inside $\D(\A)$, we then have
\[
\D_{\tn{fp}}(\A)=\Perf(\A).
\]
Moreover, $\A$ has perfect dimension $\le q$.
\end{proposition}
\begin{proof} The hypotheses imply $\D_{\tn{fp},\pdim \le q}(\A)=\D_{\tn{fp}}(\A)$ so that the whole sequence \zcref{eq:bigdiagram} collapses! This yields what we want.
\end{proof}

Under suitable conditions, we can avoid the idempotent completion $\Split(-)$ in \zcref{lem:perfect_dimension}.

\begin{definition}
Let $\A$ be a graded category (i.e.\ $\A$ is enriched in graded abelian groups).
We say that $\A$ has \emph{free fp-projectives} if every finitely presented graded projective $\A$-module is a finite sum of shifts of representable functors.
\end{definition}
\begin{corollary} \label{cor:pretriang}
Let $\A$ be a homologically coherent $A_\infty$-category
  such that $H^\ast(\A)$ has finite  global dimension~$q$. Assume in addition that $H^\ast(\A)$ has free fp-projectives.  Then $\Tw_{\le q} \A$ is pretriangulated and 
there is an equivalence of triangulated categories
\begin{equation}
\label{eq:nosplit}
H^0(\operatorname{Tot}):H^0(\Tw_{\le q} \A) \rightarrow \Perf(\A).
\end{equation}
\end{corollary}
\begin{proof}
  Put $\A^{\circ q+1} = \tn{Free}(\A) \star \cdots\star \tn{Free}(\A)$ with $q+1$ factors (cf. \zcref{lem:perfect_dimension}). Then the initial step of the proof of the leftmost inclusion in \zcref{eq:bigdiagram} can be modified to obtain
\[
\D_{\tn{fp},\pdim\le q}(\A)\hookrightarrow \A^{\circ q+1}.
\]
Since \zcref{eq:bigdiagram} collapses, we obtain an equivalence $\A^{\circ q+1}\rightarrow \Perf(\A)$. We now obtain \eqref{eq:nosplit} from the observation
that $H^0(\operatorname{Tot})$ induces an equivalence $H^0(\Tw_{\le q} \A)  \cong \A^{\circ q+1}$.
Since $\Perf(\A)$ (with its enhancement inherited from $\C_\infty(\A)$) is pretriangulated, so is $\Tw_{\le q} \A$.
\end{proof}

\section{A triangulated equivalence for \texorpdfstring{$N$}{N}-formal \texorpdfstring{$A_{\infty}$}{Ainfty}-categories} \label{section - 2}
\subsection{\texorpdfstring{$N$}{N}-formality}
Recall that an $A_{\infty}$-category $\A$ is \ti{formal} if there is an $A_\infty$-morphism $f:\A\rightarrow (H^\ast(\A),m_2)$ such that $H^\ast(f_1)$ is the identity on $H^\ast(\A)$.
Likewise, we make the following definition:
\begin{definition} \label{def - N-formal}
    Let $N > 2$. An $A_{\infty}$-category $\A$ is $N$-\ti{formal} if there is an $A_\infty$-morphism $f:\A\rightarrow (H^\ast(\A),(m_n)_n)$ 
such that $H^\ast(f_1)$ is the identity on $H^\ast(\A)$ and such that $m_3,\ldots,m_N=0$.
\end{definition}

% 2.1 - THE UNDIRECTED CASE
\subsection{The undirected case} \label{subsection - 2.1} 

In this section, we prove the triangulated equivalence of 
\zcref{thm - undirected}.

% 2.1.1 N-FORMALITY

Let $f$ be as in \zcref{def - N-formal}. We will be slightly sloppy by viewing $f$ as an $A_M$-functor for any $M$. Let $f_{\le N}$ be the $A_N$-functor that is the composition
\[
\mathcal{A} \xrightarrow[A_\infty]{f}  (H^\ast(\A),(m_n)_n) \xrightarrow[A_N]{\Id}  (H^\ast(\A),m_2),
\]
where $\Id$ is the strict $A_N$-functor (cf. \zcref{def:functors}) that is the identity on objects and morphisms. By \zcref{prop - functoriality *q and Twq}, $f_{\le N}$ induces an $A_{\lfloor \frac{N-q}{q+1} \rfloor}$-functor
\begin{equation} \label{eq - Am-functor}
    \Tw_{\leq q} f_{\le N}: \Tw_{\leq q} \A \to \Tw_{\leq q}(H^{*}(\A),m_2).
\end{equation}
Also by \zcref{prop - functoriality *q and Twq}, this $A_{\lfloor \frac{N-q}{q+1} \rfloor}$-functor may be written as the composition
\begin{equation}
\label{eq:ANfactorization}
\Tw_{\leq q} \A \xrightarrow[A_\infty]{\Tw f} \Tw_{\leq q}(H^\ast(\A),(m_n)_n)\xrightarrow[A_{\lfloor \frac{N-q}{q+1} \rfloor}]{\Id}  \Tw_{\leq q}(H^\ast(\A),m_2)
\end{equation}
since clearly $\Tw_{\le q} \Id=\Id$.

The following result was stated in the introduction.
\begin{thm}\label{thm - undirected}
    Let $\mc{A}$ be an $N$-formal $k$-linear $A_{\infty}$-category for $N \geq 3$.
Assume furthermore that $H^\ast(\A)$ is right graded coherent
and that it has right finite  global dimension~$q$.
 Then there is an equivalence of triangulated categories
\begin{equation}    
\label{eq:body}
\D_{\tn{fp}}(\mc{A}) \cong \D_{\tn{fp}}(H^{*}(\mc{A})) \quad \tn{ if } N \geq 8q+7,
\end{equation}
which moreover extends the identity isomorphism between the full subcategories
of $\D_{\tn{fp}}(\A)$ and $\D_{\tn{fp}}(H^\ast(A))$ spanned by the shifts of the objects in $\A$ and $H^\ast(\A)$.
\end{thm}
\begin{proof}
Assume $N\geq 8q+7$. This implies 
\[
M:= \left \lfloor \frac{N-q}{q+1} \right \rfloor \geq 7
\]
and
\[
 \left \lfloor \frac{M-1}{1+1} \right \rfloor \geq 3.
\]
In other words, by \zcref{prop - functoriality *q and Twq}, the identity is an $A_3$-functor
\[
\Tw_{\le 1} (\Tw_{\le q} (H^\ast(\A), (m_n)_n)) \rightarrow \Tw_{\le 1} (\Tw_{\le q} (H^\ast(\A), m_2)),
\]
meaning it respects $m_2$ and $m_3$. It follows from \zcref{prop - Massey Ainfty} below that it preserves distinguished triangles. Hence, by \zcref{lem:splitlemma},
it induces an isomorphism of triangulated categories
\[
\Split(H^0(\Tw_{\leq q}(H^\ast(\A),(m_n)_n)))\xrightarrow{\Id} \Split(H^0(\Tw_{\leq q}(H^\ast(\A),m_2))).
\]
Consider the following commutative diagram
\begin{equation}
\begin{tikzcd}
\Split (H^0(\Tw_{\leq q} \A ))\arrow[d, "\cong"]\arrow[r,"{\cong}"', "{H^0(\Tw_{\le q} f)}"] & \Split(H^0(\Tw_{\leq q}(H^\ast(\A),(m_n)_n)))
\arrow[d, "\cong"]\arrow[r, "{\Id}"] &\arrow[d, "\cong"] \Split(H^0(\Tw_{\leq q}(H^\ast(\A),m_2)))\\
\D_{\fp}(\A)\arrow[r, dotted] & \D_{\fp}(H^\ast(\A), (m_n)_n)\arrow[r, dotted]& \D_{\fp}(H^\ast(\A), m_2),
\end{tikzcd}
\end{equation}
where the vertical arrows are derived from \zcref{prop - Tot qequiv} and \zcref{lem:perfect_dimension}. We now obtain \zcref{eq:body} as the composition
of the dotted arrows.
\end{proof}

By strengthening the hypotheses of \zcref{thm - undirected}, we can relax the bound in \eqref{eq:body}.
\begin{thm} \label{thm - undirected improved} Let the hypotheses be as in \zcref{thm - undirected}, and assume in addition that $H^\ast(\A)$ has free fp-projectives. Then \zcref{eq:body} holds
with the bound $N\geq 4q+3$.
\end{thm}
\begin{proof} If $H^\ast(\A)$ has free fp-projectives, then $\Tw_{\le q} \A$, $\Tw_{\le q}(H^\ast(\A),(m_n)_n)$ and $\Tw_{\le q}(H^\ast(\A),m_2)$ are pretriangulated and
idempotent complete by \zcref{cor:pretriang}. Hence, there is no need to take split closures in the proof of \zcref{thm - undirected}, and it is enough
that $\tn{Id}: \Tw_{\le q} (H^\ast(\A), (m_n)_n)\rightarrow \Tw_{\le q}(H^\ast(\A), m_2)$ is an $A_3$-functor. This follows from the fact
that $N\ge 4q+3$ implies that $\lfloor (N-q)/(q+1)\rfloor \ge 3$.
\end{proof}

% 2.2 - THE DIRECTED CASE
\subsection{The directed case}
In this section, we improve the lower bound on $N$ in \zcref{thm - undirected} by restricting to the directed setup, thereby obtaining \zcref{thm - directed}.

% 2.2.1 - DIRECTED GRADED CATEGORIES
\begin{definition}
\label{def:directed}
Let $K$ be a field extension of $k$. A $K$-linear graded category ${{\A}}$ is said to be \emph{directed} of length $l$ if $|\Ob({{\A}})|<\infty$, $\dim_K {{\A}}(A,B)<\infty$ for all $A,B\in \Ob({{\A}})$,
and if there is a decomposition $\Ob({{\A}})=O_0\coprod\cdots \coprod O_{l}$ such that
\begin{enumerate}
\item if $A\in O_i$, $B\in O_j$, $j<i$, then ${{\A}}(A,B)=0$;
\item all non-zero maps between objects in $O_i$ for $0 \leq i \leq l$ are isomorphisms. 
\end{enumerate}
We refer to the $(O_i)_i$ as \emph{blocks}.
\end{definition}
\begin{remark} We consider blocks and the field $K$ to be part of the structure of $\A$, but we will usually mention them only when necessary.
\end{remark}
\begin{lemma} \label{lem:blockform}
Let ${{\A}}$ be a minimal $A_\infty$-category such that $H^\ast({{\A}})$ is directed\footnote{By definition, $H^\ast({{\A}})$ is linear over a field $K$. However
  we do not assume that the higher operations on $\A$ are $K$-linear. Following our general conventions (see \zcref{sec:notation_and_conventions}) we assume of course that the
  higher operations are $k$-linear.} of length $l$. Let $M\in \D_{\fp}({{\A}})$ be such that $H^\ast(M(X))=0$ for 
$X\in \bigcup_{i\ge k} O_{i}$.
Then $M$ is isomorphic to an object of the form $\tn{Tot}((A,\delta))$, where $A=\bigoplus_{i=0}^{k-1} A_i$ and $A_i$ is a finite sum of $\Sigma^j {{\A}}(-,B)$ for $j\in \mathbb{Z}$, $B\in O_i$
such that in addition $\delta_{ij}:A_j\rightarrow A_i$ is zero for $j\ge i$.
\end{lemma}
\begin{proof}
We use descending induction on $k$. 
If $k=0$, then $M=0$ and there is nothing to prove.

So assume $k>0$. Let $B_1,\ldots,B_u$ be representatives of the isomorphism classes of objects in 
$O_{k-1}$. We have ${{\A}}(B_p, B_q)=0$ if $p\neq q$, and $K_p:={{\A}}(B_p,B_p)$ is a field.
Then $H^\ast(M(B_p))=\operatorname{Hom}^\ast_{\D({{\A}})}(\Y(B_p),M)$ is a finite
dimensional graded
$K_p$-vector space.

For each $p$, there is an induced map $\operatorname{Hom}^\ast_{\D({{\A}})}(\Y(B_p),M)
\otimes_{K_p} \Y(B_p)\xrightarrow{\alpha_p}
M$. Consider the distinguished triangle in $\D({{\A}})$:
\begin{equation}
\label{eq:blocks}
N\xrightarrow{\beta} \bigoplus_p\operatorname{Hom}^\ast_{\D({{\A}})}(\Y(B_p),M)
\otimes_{K_p} \Y(B_p) \xrightarrow{(\alpha_p)_p} M\rightarrow \Sigma N.
\end{equation}
Then one sees that $H^\ast(N(X))=0$ for $X\in \bigcup_{i\ge k-1}O_i$. By
induction, we then obtain $N\cong \operatorname{Tot}((A',\delta'))$, where
$\delta'$ is as asserted in the statement of the lemma but with $k$ replaced by $k-1$.
Moreover, the middle term of \zcref{eq:blocks} is of the form 
$\operatorname{Tot}((A'',0))$ where $A''\in \operatorname{Free}(\A)$ is a sum of shifts of $(B_p)_p$. Finally, $\beta$ can be viewed as a closed map of degree zero
$(A',\delta')\xrightarrow{\beta} (A'',0)$ in 
$\Tw \A$. It then follows that $\operatorname{cone} (\beta)\cong (A,\delta)$ (cf. \zcref{eq:cone}) in $H^0(\Tw \A)$ with
$(A,\delta)$ as asserted in the theorem.
\end{proof}

This lemma leads to the following definition.

\begin{definition}
Let ${{\A}}$ be a minimal $A_\infty$-category such that $H^\ast({{\A}})$ is directed
of length $l$. Then 
$\Tw^B_{\le l} \A$ is the full $A_\infty$-subcategory of $\Tw \A$ consisting of objects $(A,\delta)$ as in the statement of \zcref{lem:blockform} with $k=l+1$.
\end{definition}

\begin{corollary} Let ${{\A}}$ be a minimal $A_\infty$-category such that $H^\ast({{\A}})$ is directed of length $l$. Then
$\Tw^B_{\le l} \A$ is pretriangulated. Moreover, we have the following equivalences inside $\D(\A)$,
$$H^0(\Tw^B_{\le l} \A)\cong \D_{\operatorname{fp}}(\A)\cong \Perf(\A).$$
\end{corollary}
\begin{proof}
It follows from \zcref{lem:blockform} that $H^0(\Tw^B_{\le l} \A)\cong \D_{\operatorname{fp}}(\A)$, from which we may in particular deduce that $\Tw^B_{\le l}\A$
is pretriangulated. As in \zcref{eq:bigdiagram}, we have $\Perf(\A)\subset \D_{\operatorname{fp}}(\A)$. The fact that the essential image of $H^0(\Tw^B_{\le l}\A)$ 
lies in $\Perf(\A)$ finishes the proof.
\end{proof}

The following is our main result in this section.

\begin{proposition} Let ${{\A}}$ be a minimal $A_\infty$-category such that $H^\ast({{\A}})$ is directed of length $l$ and let $\B$ be an arbitrary $A_\infty$-category.
Let $f:\A\rightarrow \B$ be an $A_n$-functor for $n\ge l$. Then $f$ induces an $A_{n-l}$-functor  $\Tw^B_{\le l}f:\Tw^B_{\le l}\A \rightarrow \Tw \B$.
\end{proposition}
\begin{proof} We need to prove that $\Tw^B_{\le l}f$ satisfies the $A_\infty$-identities for sequences of (shifted) morphisms in  $\Tw^B_{\le l}f$  with at most $d=n-l$ elements.
As in the proofs of \cite[Lemmas 6.3, 6.4]{RVdB}, evaluations on such sequences expand to evaluations on composable sequences of morphisms in $\A$ of the form
\begin{equation}
\label{eq:delta_expression}
\underbrace{(s\delta_{d})_{\bullet\bullet}\otimes\cdots\otimes (s\delta_{d})_{\bullet\bullet}}_{i_d}\otimes(sa_{d})_{\bullet\bullet}\otimes\cdots\otimes(sa_2)_{\bullet\bullet} \otimes\underbrace{ (s\delta_{2})_{\bullet\bullet}\otimes\cdots\otimes 
(s\delta_1}_{i_1})_{\bullet\bullet}\otimes(sa_{1})_{\bullet\bullet}\otimes\underbrace{(s\delta_{0})_{\bullet\bullet}\otimes\cdots\otimes(s\delta_0)_{\bullet\bullet}}_{i_0},
\end{equation}
where $(..)_{\bullet\bullet}$ denotes a component.
However, the non-zero $(a_i)_{\bullet\bullet}$ do not decrease the block number while the non-zero $(s\delta)_{\bullet\bullet}$ strictly increase it. Hence, 
\zcref{eq:delta_expression} is zero when $i_0+\cdots+i_d> l$. So it is sufficient to have the $A_n$-identities for $f$ when $d+l\ge n$, which is precisely our hypothesis.
\end{proof}

We can now improve the lower bound of \zcref{thm - undirected} using directedness.

\begin{thm} \label{thm - directed}
  Let $\mc{A}$ be an $N$-formal $k$-linear $A_{\infty}$-category for $N \geq 3$ such that $H^\ast(\A)$
is directed  of length~$l$. Then \zcref{thm - undirected} holds with the bound \zcref{eq:body} replaced by $N\ge l+3$.
\end{thm}
\begin{proof} We obtain the analogue of \zcref{eq:ANfactorization} except that the identity map is now an $A_{N-l}$-functor. Since $N\ge l+3$, 
it is still an $A_3$-map, and we may
proceed as before.
\end{proof}

% ----------------------------------------------------------------------------------------
% 3 - EXAMPLES OF N-FORMAL AINFTY-CATS VIA HOCHSCHILD COCYCLES
\section{Examples} \label{sec - examples}
In this section, we use the previous results to construct two new examples of triangulated categories with non-unique enhancements. In contrast to the examples in \cite{RVdBUnic}, the new examples are equipped with a $t$-structure.

\subsection{\texorpdfstring{$A_\infty$}{Ainfty}-algebras with hereditary cohomology}
Below, we discuss some examples where $\A$ is an $A_\infty$-category such that $H^\ast(\A)$ is hereditary. Note that this implies in particular
that $H^\ast(\A)$ is coherent, so we are in the settting of \zcref{prop - Tot qequiv} with $q=1$. We recall the following result:
\begin{thm}[{\cite[Theorem 3.6]{KYZ}}] The functor \label{thm:keller}
  \[
    H^\ast(-):\Perf(\A) \rightarrow \tn{gr}(H^\ast(\A)):M\mapsto H^\ast(M)
  \]
(where $\tn{gr}(-)$ denotes the category of finitely graded modules)  is full, essentially surjective, and its kernel has square zero. In
particular, it induces a bijection between the set of isoclasses of objects (respectively, of indecomposable objects) in $\Perf(\A)$ and $\tn{gr}(H^\ast(\A))$.
\end{thm}
\subsection{Example 1} \label{subsec - ex1}
In the rest of this section, $k$ is a field of characteristic zero\footnote{In this case, we can use the Hochschild-Konstant-Rosenberg theorem (see for example 
  \cite[Theorem 9.1.3]{Ginzburg}) to conclude that $\tn{HH}^{i}_{k}(K) \cong \Lambda^{i}\tn{Der}(K) \cong K \neq 0$ for $K := k(x_{1},\ldots,x_{m})$ the fraction field of
  $k[x_{1},\ldots,x_{m}]$.}.
Let $N \geq 4$ and define $K := k(x_{1},\ldots,x_{N+1})$ as the fraction field of $k[x_{1},\ldots,x_{N+1}]$. We fix an $(N+1)$-Hochschild cocycle $\eta$ such that $0 \neq [\eta] \in \tn{HH}^{N+1}_{k}(K)$. 

Via the Yoneda Ext-interpretation, together with \zcref{eq - Ainf dg bimod} and \zcref{eq - dg qiso U}, we obtain that 
\begin{align*}
    \tn{HH}^{N+1}_{k}(K,K) &\cong \Ext^{N+1}_{K \otimes_{k} K}(K,K) \\
    &= \mc{D}(U(K),U(K))(K,K[N+1]) \\
    &\cong \mc{D}(K,K)(K,K[N+1]),
\end{align*}
so we can interpret $\eta$ as an $A_{\infty}$-$K$-$K$-bimodule morphism $K \to K[N+1]$. We then define $M \in \mc{D}(K,K)$ as
\[\begin{tikzcd}
	M & K & {K[N+1]} & {\Sigma M,}
	\arrow[from=1-1, to=1-2]
	\arrow["\eta", from=1-2, to=1-3]
	\arrow["{}", from=1-3, to=1-4]
\end{tikzcd}\]
i.e. $M := \Cone(\eta)[-1]$. By considering the long exact sequence, we find
\begin{equation*}
    H^{i}(M) = \begin{cases} K, &j = 0 \text{ or } j = -N, \\
    0, &\tn{elsewhere.} \end{cases}
\end{equation*}
Thus, $H^{*}(M) = K \oplus K[N]$. We may replace $M$ by a minimal $A_{\infty}$-$K$-$K$-bimodule and can then use $M$ to construct a minimal $A_{\infty}$-category out of $K$. This will be (an $A_{\infty}$-variation on) the lower triangular matrices of 
\cite[\S 4.5]{KellerDerivedHochschild} or \cite[\S 3.2]{OrlovGlue}, or the arrow category of \cite[\S 9.2]{RVdB}.

\begin{definition} 
  Let $\mf{g}_{1} = K \xrightarrow{M} K$ be the $A_{\infty}$-category with as objects $\{ A, B\}$ and $\tn{Hom}$-spaces
    \begin{align*}
        &\mf{g}_{1}(A,A) = \mf{g}_{1}(B,B) = K, \\
        &\mf{g}_{1}(A,B) = M, \\
        &\mf{g}_{1}(B,A) = 0.
    \end{align*}
\end{definition}

By construction, we immediately obtain:

\begin{proposition} \label{prop - ex1}
    The following hold:
    \begin{enumerate}
        \item[$(i)$] $\mf{g}_{1}$ is an $N$-formal $k$-linear $A_{\infty}$-category that is not formal;
        \item[$(ii)$] $H^\ast(\mf{g}_{1})$ is directed of length $1$; 
        \item[$(iii)$] $\mf{g}_{1}$ is cohomologically concentrated in nonpositive degrees.
    \end{enumerate}
  \end{proposition}
  \begin{proof} We only address $(i)$. If $\mf{g}_{1}$ were formal, then $M$ would be formal in $\D(K,K)$. However, this contradicts the construction
    of $M$ as $[\eta] \neq 0$.
    \end{proof}

Since $N \geq 4$,
$$N - l \geq 4-1 = 3 \geq 3.$$
Consequently, \zcref{thm - directed} applies, so there is an equivalence of $k$-linear triangulated categories
\begin{equation}
  \label{eq:2enhancements}
  \D_{\tn{fp}}(\mf{g}_{1}) \cong \D_{\tn{fp}}(H^{*}(\mf{g}_{1})).
  \end{equation}
This equivalence induces two enhancements on $\D_{\tn{fp}}(H^{*}(\mf{g}_{1}))$: the standard one and the obtained by transferring from $\D_{\tn{fp}}(\mf{g}_{1})$ (cf. \zcref{sec:idempotent_completion}).
\begin{proposition}
  \label{prop:2enhancements}
  The two enhancements on $\D_{\tn{fp}}(H^{*}(\mf{g}_{1}))$ obtained from \zcref{eq:2enhancements} are not equivalent. In other words,
  $\D_{\tn{fp}}(H^{*}(\mf{g}_{1}))$ does not have a unique enhancement.
\end{proposition}
\begin{proof} Let $\C$ be the standard enhancement on $\D_{\tn{fp}}(\mf{g}_{1})$ (see \zcref{sec:enhancements})
  and let $\C'$ be the standard dg-enhancement on
  $\D_{\tn{fp}}(H^{*}(\mf{g}_{1}))$ given by semi-free resolutions. Note that $\C'$ is $K$-linear.
  If  the two enhancements on  $\D_{\tn{fp}}(H^{*}(\mf{g}_{1}))$ are equivalent, then there is an $A_\infty$-quasi-equivalence
  $F:\C\rightarrow \C'$. 
  Then $F$ induces an $A_\infty$-isomorphism between $\mf{g}_{1}$ and the full subcategory of $\C'$ spanned by $F(A),F(B)$. Since $F(A), F(B)$ is an exceptional
  sequence (being the image of an exceptional sequence in $\D_{\tn{fp}}(\A)$), this category is formal by \zcref{lem:formal} below. This contradicts the fact that $\mf{g}_1$ is not formal.
\end{proof}

\begin{lemma}
  \label{lem:formal}
  Let $\A$ be a $K$-linear $A_\infty$-category given by an exceptional collection of length two, that is, $\Ob(\A)=\{E,F\}$ and
  \begin{enumerate}
  \item
    The maps $K\rightarrow \A(E,E)$, $K\rightarrow \A(F,F)$ are quasi-isomorphisms;
  \item $\A(F,E)$ is acyclic.
  \end{enumerate}
  Then $\A$ is formal.
\end{lemma}
\begin{proof}
  We may replace $\A$ by a $K$-linear minimal model. Then, by $K$-linearity and the fact
  that $\A$ is strictly unital, we see that $\A=K\xrightarrow{M} K$, where
  $M=\A(E,F)$ with no higher operations. Since $\A$ has zero differential, we are done.
\end{proof}
\subsection{Example 2} \label{subsec - ex2}
 We now let $K := k(x_{1},\ldots,x_{N})$ be the fraction field of $k[x_{1},\ldots,x_{N}]$. We fix an $N$-Hochschild cocycle $\eta$ such that $0 \neq [\eta] \in \tn{HH}^{N}_{k}(K)$.

\begin{definition}\label{def - perturbed}
Let $N \geq 7$ be odd\footnote{This ensures graded commutativity of $K[t]$ below, since $\vert t \vert = 1-N$ will have even degree.}.
    Consider the graded algebra $K[t]$ where $\vert t \vert = 1-N$. We endow it with a minimal $A_{\infty}$-structure using an argument similar to the one in \cite{RVdBUnic}:
    Recall that, by the graded HKR-theorem (see for example \cite[Theorem 9.1.3]{Ginzburg}),
    \begin{align*}
        \tn{HH}^{i}_{k}(K) &= 0 \qquad \tn{ if } i > N, \\
        \tn{HH}^{i}_{k}(k[t]) &= 0 \qquad \tn{ if } i > 1, \\
        \tn{HH}^{1}_{k}(k[t]) &= k[t] d/dt.
    \end{align*}
    By \cite[Lemma 1]{RVdBUnic}, we have that
    $$\tn{HH}^{*}_{k}(K[t]) \cong \tn{HH}^{*}_{k}(K) \otimes \tn{HH}^{*}_{k}(k[t]).$$
    As a result, we have that $\tilde{\eta} := \eta \otimes d/dt$ defines a non-trivial element of $\tn{HH}^{N+1}_{k}(K[t])$. We can also conclude that $\tn{HH}^{i}_{k}(K[t]) = 0$ for $i > N+1$, which means that, by \cite[B.4.1]{Lefevre}, there are no obstructions to extending $\tilde{\eta}$ to a minimal $A_{\infty}$-structure on $K[t]$. So we may construct a minimal $A_{\infty}$-structure 
    $$(0,m_{2},0,\ldots,0,m_{N+1},m_{N+2},\ldots)$$
    on $K[t]$ so that $m_{2}$ is the original multiplication and the class of $m_{N+1}$ is $[\tilde{\eta}]$. We denote the resulting minimal $A_{\infty}$-algebra by $\mf{g}_{2}$. The corresponding $A_{\infty}$-category with one object will also be denoted by $\mf{g}_{2}$.
\end{definition}

\begin{proposition} \label{prop - ex2}
    The following hold:
    \begin{enumerate}
        \item[$(i)$] $\mf{g}_{2}$ is an $N$-formal $k$-linear $A_{\infty}$-category that is not formal;
        \item[$(ii)$] $\mf{g}_{2}$ is homologically coherent and $H^{*}(\mf{g}_{2}) = K[t]$ is hereditary;
        \item[$(iii)$] $H^\ast(\mf{g}_{2})$ has free fp-projectives;
        \item[$(iv)$] $\mf{g}_{2}$ is cohomologically concentrated in nonpositive degrees.
    \end{enumerate}
    \begin{proof}
        That $\mf{g}_{2}$ is not formal follows by \cite[B.4.2]{Lefevre}: Any $A_{\infty}$-isomorphism $f: \mf{g}_{2} \to H^{*}(\mf{g}_{2})$ has $f_{1}$-component in $\tn{Aut}_{k}(K[t])$. By \cite[B.4.2]{Lefevre}, the obstruction to extending an arbitrary automorphism $f_{1} \in \tn{Aut}_{k}(K[t])$ to an $A_{\infty}$-isomorphism is given by $\tilde{\eta} \circ f_{1}$. Since $\tilde{\eta}$ is not trivial, neither is this obstruction, and we are done.

        The fact that $H^\ast(\mf{g}_{2})$ has free fp-projectives follows from the fact that $K[t]$ is a PID so that projectives are free.
    \end{proof}
\end{proposition}

Since $N \geq 7$,
$$N - 4q \geq 7-4 = 3 \geq 3.$$
Consequently, \zcref{thm - undirected} applies, so there is an equivalence of $k$-linear triangulated categories
\begin{equation}
  \label{eq:2enhancements2}
\D_{\tn{fp}}(\mf{g}_{2}) \cong \D_{\tn{fp}}(H^{*}(\mf{g}_{2})).
\end{equation}
\begin{proposition}
  \label{prop:2enhancements2}
  The two enhancements on $\D_{\tn{fp}}(H^{*}(\mf{g}_{2}))$ obtained from \zcref{eq:2enhancements2} are not equivalent. In other words, $\D_{\tn{fp}}(H^{*}(\mf{g}_{2}))$ does not have a unique enhancement.
\end{proposition}
\begin{proof}
  We use the same notations as in the proof of Proposition \ref{prop:2enhancements}. $F$ induces an $A_\infty$-isomorphism between $\mf{g}_2$ and the category
  spanned by $F(\mf{g}_2)$. Now, $F(\mf{g}_2)$ is an indecomposable object in $\D_{\tn{fg}}(H^\ast(\mf{g}_2))$ and, since $K[t]$ is a PID, it follows from
  \zcref{thm:keller} that $F(\mf{g}_2)$ is a shift of $H^\ast(\mf{g}_2)$. Now $\C'(\Sigma^l H^\ast(\mf{g}_2),\Sigma^l H^\ast(\mf{g}_2))=K[t]$, which is formal. This contradicts the non-formality of $\mf{g}_2$.
  \end{proof}

\appendix

% ----------------------------------------------------------------------------------------
% 2 - MASSEY PRODUCTS IN A-INFINITY CATEGORIES
\section{The triangulated structure of \texorpdfstring{$A_\infty$}{Ainfty}-categories} \label{appendix}
\subsection{Main result}
For the convenience of the reader we give a proof of the following result:
\begin{proposition} \label{prop - Massey Ainfty}
  Let $\B$ be an $A_\infty$-category. Then the distinguished triangles in $H^\ast(\B)$ in the sense of \zcref{def - cone}
are determined by the operations $m_{1},m_{2}$ and $m_{3}$.
Namely, a sequence of composable morphisms in $H^\ast(\B)$, 
    \[\begin{tikzcd}
	X & Y & Z & {X,}
	\arrow["f", from=1-1, to=1-2]
	\arrow["g", from=1-2, to=1-3]
	\arrow["h", from=1-3, to=1-4]
	\arrow["{(1)}"', from=1-3, to=1-4]
\end{tikzcd}\]
satisfying $gf = 0 = hg$, yields a distinguished triangle in $H^0(\B)$ if and only if the set of Massey products  $\langle h,g,f \rangle \subseteq H^{0}(\B(X,X))$ as defined in
\zcref{def - Massey Ainf} below, contains $\tn{id}_X$.
\end{proposition}

\begin{corollary} \label{cor - Massey Ainfty}
  The triangulated structure of a pretriangulated $A_{\infty}$-category $\B$ is determined by the operations $m_{1},m_{2}$ and $m_{3}$.
\end{corollary}

This result is proven by comparing Massey products in triangulated categories and in $A_\infty$-categories.

\subsection{Massey products in triangulated categories}
\label{sec:massey_triang}
Below, we will review some results from \cite[\S1.3]{Heller} regarding Massey products in triangulated categories. Let $\mc{T}$ be a triangulated category with translation functor $\Sigma$ and assume there are composable maps in $\T$,
\[\begin{tikzcd}
	X & Y & Z & {U,}
	\arrow["f", from=1-1, to=1-2]
	\arrow["g", from=1-2, to=1-3]
	\arrow["h", from=1-3, to=1-4]
	\arrow["{(1)}"', from=1-3, to=1-4]
\end{tikzcd}\]
satisfying $gf = 0 = hg$. By the convention \zcref{sec:enrichment}, we formally have $\T^1(Z,U)=\T(Z,\Sigma U)$, so
$h$ may also be viewed as a map $Z\rightarrow \Sigma U$.

\begin{definition} \label{def - Massey triang}
    The \ti{set of Massey products $\langle h,g,f \rangle \subseteq \T(X, U)$} is the (non-empty) set of morphisms $b \in \T(X, U)$ completing the following commutative diagram:
    \begin{equation} \label{eq - triangle Massey} \begin{tikzcd}
	X & Y & Z & U \\
	X & Y & C & {X,}
	\arrow["f", from=1-1, to=1-2]
	\arrow[equals, from=1-1, to=2-1]
	\arrow["g", from=1-2, to=1-3]
	\arrow[equals, from=1-2, to=2-2]
	\arrow["h", from=1-3, to=1-4]
	\arrow["{(1)}"', from=1-3, to=1-4]
	\arrow["f", from=2-1, to=2-2]
	\arrow["i", from=2-2, to=2-3]
	\arrow["a", from=2-3, to=1-3]
	\arrow["p", from=2-3, to=2-4]
	\arrow["b", from=2-4, to=1-4]
	\arrow["{(1)}"', from=2-3, to=2-4]
\end{tikzcd}\end{equation}
    where the lower row is a distinguished triangle in $\T$.
\end{definition}

\begin{proposition}
    The set $\langle h,g,f \rangle$ is a coset of $h \T^{-1}(X,Z) + \T(Y,U)  f$ inside $\T(X,U)$. 
\end{proposition}

Massey products allow us to detect distinguished triangles in $\T$.

\begin{proposition} \label{prop - Classic Massey}
    A series of composable arrows
    \[\begin{tikzcd}
	X & Y & Z & X
	\arrow["f", from=1-1, to=1-2]
	\arrow["g", from=1-2, to=1-3]
	\arrow["h", from=1-3, to=1-4]
	\arrow["{(1)}"', from=1-3, to=1-4]
\end{tikzcd}\]
    in $\T$ satisfying $gf = 0 = hg$ is a distinguished triangle if and only if $\id_{X} \in \langle h,g,f \rangle$.
\end{proposition}

\subsection{Massey products in \texorpdfstring{$A_\infty$}{Ainfty}-categories}

Let $\B$ be an $A_{\infty}$-category. Suppose there are composable morphisms in $H^{0}(\B)$,
\[\begin{tikzcd}
	X & Y & Z & {U,}
	\arrow["f", from=1-1, to=1-2]
	\arrow["g", from=1-2, to=1-3]
	\arrow["h", from=1-3, to=1-4]
	\arrow["{(1)}"', from=1-3, to=1-4]
\end{tikzcd}\]
satisfying $gf = 0 = hg$.

\begin{definition} \label{def - Massey Ainf}
    The \ti{set of Massey products $\langle h,g,f \rangle \subseteq H^{0}(\B(X,U))$} is the (non-empty) set of elements $b \in H^{0}(\B(X,U))$ obtained via the following procedure:
    \begin{enumerate}
        \item Choose lifts $\tilde{f},\tilde{g}$ in $Z^{0}(\B)$, and $\tilde{h}\in Z^{1}(\B)$, of $f,g,h$ respectively.
        \item Since $gf = 0$ in $H^{0}(\B)$, there is some $u \in \B(X,Z)^{-1}$ such that $$b_{1}(su) = -b_{2}(s\tilde{g},s\tilde{f}).$$ Similarly, there is some $v \in \B(Y,U)^{0}$ for which $$b_{1}(sv) =b_{2}(s\tilde{h},s\tilde{g}).$$ Note that the signs are chosen such that $m_1(u)=m_2(\tilde{g}, \tilde{f})$ and $m_1(v)=m_2(\tilde{h},\tilde{g})$, following the usual convention for Massey products. 
        \item Set $$s\tilde{b} := -b_{2}(s\tilde{h},su)  + b_{2}(sv,s\tilde{f})  -b_{3}(s\tilde{h},s\tilde{g},s\tilde{f}).$$ Then, by \zcref{lem - def of btilde}, we have $b_1(s\tilde{b})=0$.
        \item Hence $\tilde{b} \in Z^{0}(\B(X,U))$. Let $b$ be the image of $\tilde{b}$ in $H^{0}(\B(X,U))$.
    \end{enumerate}
\end{definition}

\begin{lemma}\label{lem - def of btilde}
With the notation of \zcref{def - Massey Ainf}, we have  $b_1(s\tilde{b})=0$.
\end{lemma}
\begin{proof}
By the Koszul sign rule and the fact that $\tilde{f},\tilde{g},\tilde{h}$ are closed, we find that
        \begin{align*}
            &\quad b_{1}(s\tilde{b}) \\
            &= -b_{1}b_{2}(s\tilde{h},su)+b_{1}b_{2}(sv,s\tilde{f}) - b_{1}b_{3}(s\tilde{h},s\tilde{g},s\tilde{f}) \\
            &= b_{2}(b_{1} \otimes 1+1 \otimes b_{1})(s\tilde{h},su) 
            -b_{2}(b_{1} \otimes 1+1 \otimes b_{1})(sv,s\tilde{f}) 
            -b_{1}b_{3}(s\tilde{h},s\tilde{g},s\tilde{f})\\
            &= b_{2}(s\tilde{h},b_{1}(su)) -b_{2}(b_{1}(sv),s\tilde{f}) -b_{1}b_{3}(s\tilde{h},s\tilde{g},s\tilde{f})\\
            &= -b_{2}(s\tilde{h},b_{2}(s\tilde{g},s\tilde{f})) - b_{2}(b_{2}(s\tilde{h},s\tilde{g}),s\tilde{f}) - b_{1}b_{3}(s\tilde{h},s\tilde{g},s\tilde{f}) \\
            &= -((b_{2}(1 \otimes b_{2} + b_{2} \otimes 1) + b_{1}b_{3})(s\tilde{h},s\tilde{g},s\tilde{f})) \\
            &= 0. \qedhere
        \end{align*}
\end{proof}

\begin{proposition}
 The set $\langle h,g,f \rangle$ is a coset of $h H^{-1}(\B(X,Z)) + H^{0}(\B(Y,U)) f$ inside $H^{0}(\B(X,U))$. 
    \begin{proof}
        We start by showing that different choices of lifts $\tilde{f}',\tilde{g}',\tilde{h}'$ in step (1) and different choices $u',v'$ in step (2) of \zcref{def - Massey Ainf}, yield elements of the coset $b +h H^{-1}(\B(X,Z)) + H^{0}(\B(Y,U)) f$.
        
        We first let $\tilde{f}$ vary. So consider another lift $\tilde{f}'$. Then 
        $$s\tilde{f}' = s\tilde{f} + b_{1}(s\gamma_{f})$$
        for some $\gamma_{f} \in \B(X,Y)^{-1}$. Then
        \begin{align*}
            b_{2}(s\tilde{g},s\tilde{f}') &= b_{2}(s\tilde{g},s\tilde{f}) + b_{2}(s\tilde{g},b_{1}(s\gamma_{f})) \\
            &= -b_{1}(su) + b_{1}b_{2}(s\tilde{g},s\gamma_{f})
        \end{align*}
        by \zcref{eq - Ainfty b} and the fact that $\tilde{g}$ is closed. Setting $su' := su + b_{2}(s\tilde{g},s\gamma_{f})$, we find
        \begin{align*}
            s\tilde{b}' :&= -b_{2}(s\tilde{h},u') + b_{2}(sv,s\tilde{f}') - b_{3}(s\tilde{h},s\tilde{g},s\tilde{f}') \\
            &= s\tilde{b} - b_{2}(s\tilde{h},b_{2}(s\tilde{g},s\gamma_{f})) + b_{2}(sv,b_{1}(s\gamma_{f})) - b_{3}(s\tilde{h},s\tilde{g},b_{1}(s\gamma_{f})) \\
            &= s\tilde{b} + b_{2}(1 \otimes b_{2})(s\tilde{g},s\gamma_{f}) +b_{2}(1 \otimes b_{1})(su,s\gamma_{f}) + b_{3}(1 \otimes 1 \otimes b_{1})(s\tilde{h},s\tilde{g},s\gamma_{f}) \\
            &= s\tilde{b} - b_{1}b_{2}(su,s\gamma_{f}) - b_{1}b_{3}(s\tilde{h},s\tilde{g},s\gamma_{f}).
        \end{align*}
      We conclude that $\tilde{b}' \equiv \tilde{b}$ in $H^{0}(\B( X,U))$. Similar arguments treat the case of varying the lifts $\tilde{g}$ and $\tilde{h}$. If we now consider other choices $u' \in \B(X,Z)^{-1}$ and $v' \in \B(Y,U)^{0}$ such that $b_{1}(su') = -b_{2}(s\tilde{g},s\tilde{f})$ and $b_{1}(sv') = b_{2}(s\tilde{h},s\tilde{g})$, then
        \begin{align*}
            u' &= u + \xi_{u}, \\
            v' &= v + \xi_{v},
        \end{align*}
        for some $\xi_{u} \in Z^{-1}(\B(X,Z))$ and $\xi_{v} \in Z^{0}(\B(Y,U))$.
        Consequently,
        \begin{align*}
            s\tilde{b}' :=& -b_{2}(s\tilde{h},su') + b_{2}(sv',s\tilde{f}) - b_{3}(s\tilde{h},s\tilde{g},s\tilde{f}) \\
            =& \textcolor{white}{i} \tilde{b} - b_{2}(s\tilde{h},s\xi_{u}) + b_{2}(s\xi_{v},s\tilde{f}),
        \end{align*}
 and its image in $H^{0}(\B(X,U))$ is contained in the coset $b + h H^{-1}(\B(X,Z)) + H^{0}(\B(Y,U)) f$.

        Conversely, consider an element of the coset:
        \begin{equation} \label{eq - coset}
            b + h\alpha + \beta f.
        \end{equation}
        Then we can lift this to
        $$s\tilde{b} - b_{2}(s\tilde{h},s\tilde{\alpha}) + b_{2}(s\tilde{\beta},s\tilde{f})$$
        with $\tilde{\alpha} \in Z^{-1}(\B(X,Z))$ and $\tilde{\beta} \in Z^{0}(\B(Y,U))$. Since $\tilde{\alpha}$ and $\tilde{\beta}$ are closed, we can take $u' := u + \tilde{\alpha}$ and $v' := v + \tilde{\beta}$ in step (2) of \zcref{def - Massey Ainf}. The resulting Massey product then equals \zcref{eq - coset}.
    \end{proof}
\end{proposition}

\begin{lemma} \label{lemma - Massey Ainfty}
    The Massey products on $H^{0}(\B)$ defined in this subsection coincide with the Massey products that were discussed in \zcref{sec:massey_triang}, computed
in $H^0(\Tw \B)$.
    \begin{proof}
Let us consider the diagram of \zcref{def - Massey Ainf}, together with the maps $u$ and $v$ from the same definition. To compute its Massey product using the method of \zcref{def - Massey triang}, we can use the diagram in $\Tw \B$:
        \begin{equation} \label{eq - triangle compatibility} \begin{tikzcd}
	X & Y & Z & U \\
	X & Y & {C(f)} & { X,}
	\arrow["\tilde{f}", from=1-1, to=1-2]
	\arrow[equals, from=1-1, to=2-1]
	\arrow["\tilde{g}", from=1-2, to=1-3]
	\arrow["\tilde{h}", from=1-3, to=1-4]
	\arrow["{(1)}"', from=1-3, to=1-4]
	\arrow["\tilde{f}", from=2-1, to=2-2]
	\arrow[equals, from=2-2, to=1-2]
	\arrow["i", from=2-2, to=2-3]
	\arrow["p", from=2-3, to=2-4]
	\arrow["u", from=1-1, to=1-3, bend left]
	\arrow["v", from=1-2, to=1-4, bend left]
\end{tikzcd}\end{equation}
        where the distinguished triangle in the lower row is the standard distinguished triangle of \zcref{def - cone}. So we have
        $$C(f) = (\begin{pmatrix} \Sigma X \\ Y \end{pmatrix},\delta)$$
        with Maurer-Cartan element
        $$\delta = \begin{pmatrix}
            0 & 0 \\ \sigma^{-1}\tilde{f} & 0 
        \end{pmatrix},$$
        and the maps
        $$i = \begin{pmatrix}
            0 \\ 1
        \end{pmatrix} \quad \tn{ and } \quad p = \begin{pmatrix}
            \sigma^{-1}1 & 0
        \end{pmatrix}.$$
 
       To complete \zcref{eq - triangle compatibility} to a commutative diagram, we go in steps:

        \begin{enumerate}
            %a
            \item \tb{We construct $a$.} Define
        $$\tilde{a} := \begin{pmatrix} -\sigma^{-1}u & \tilde{g} \end{pmatrix}: C(f) \to Z,$$
	where $u$ is defined in part (2) of \zcref{def - Massey Ainf}. Then $\tilde{a}$ is closed because
            \begin{align*}
                b_{\Tw \B,1}(s\tilde{a}) &= b_{1}(s\tilde{a}) + b_{2}(s\tilde{a},s\delta) \\
                &= (b_{1}(-s\sigma^{-1}u) \quad b_{1}(s\tilde{g})) + (b_{2}(s\tilde{g},s(\sigma^{-1}\tilde{f})) \quad 0) \\
                &= (-\sigma^{-1}b_{1}(su) \quad 0) + (-\sigma^{-1}b_{2}(s\tilde{g},s\tilde{f}) \quad 0) \\
                &= (-\sigma^{-1}(b_{1}(su) +b_{2}(s\tilde{g},s\tilde{f})) \quad 0) \\
                &= (0 \quad 0),
            \end{align*}
	where the notation $b_i$ is short for $b_{\tn{Free}(\B),i}$. Let $a$ denote the image of $\tilde{a}$ in $H^{0}(\Tw \B)(C(f),Z)$. Then $a$ makes the diagram below commute:
            \[\begin{tikzcd}
	X & Y & Z & U \\
	X & Y & {C(f)} & {X.}
	\arrow["f", from=1-1, to=1-2]
	\arrow[equals, from=1-1, to=2-1]
	\arrow["g", from=1-2, to=1-3]
	\arrow["h", from=1-3, to=1-4]
	\arrow["f", from=2-1, to=2-2]
	\arrow[equals, from=2-2, to=1-2]
	\arrow["i", from=2-2, to=2-3]
	\arrow["a", from=2-3, to=1-3]
	\arrow["p", from=2-3, to=2-4]
	\arrow["u", from=1-1, to=1-3, bend left]
	\arrow["v", from=1-2, to=1-4, bend left]
\end{tikzcd}\]
            Indeed, on the one hand
            $$b_{\Tw \B,2}(s\tilde{g},s1) = s\tilde{g},$$
            while
            \begin{align*}
                b_{\Tw \B,2}&(s\tilde{a},si) = b_{2}(s\tilde{a},si) + b_{3}(s\tilde{a},s\delta,si) \\
                &= b_{2}(s(-\sigma^{-1}u \quad \tilde{g}),s(\begin{pmatrix}
                    0 \\ 1
                \end{pmatrix})) + b_{3}(s(-\sigma^{-1}u \quad \tilde{g}),s(\begin{pmatrix}
                    0 & 0 \\ \sigma^{-1}\tilde{f} & 0
                \end{pmatrix}),s(\begin{pmatrix}
                    0 \\ 1
                \end{pmatrix})) \\
                &= b_{2}(s\tilde{g},s1) \\
                &= s\tilde{g},
            \end{align*}
            where we used the fact that $b_{\tn{Free}(\B)^{\oplus 2},3}$ decomposes into expressions of $b_{\B,3}$ in which every term has a zero component.

        %b
        \item \tb{We construct $b$.} We now define $b$ following \zcref{def - Massey Ainf}, i.e. set
\[s\tilde{b} := -b_{2}(s\tilde{h},su)  + b_{2}(sv,s\tilde{f})  -b_{3}(s\tilde{h},s\tilde{g},s\tilde{f}) \]
     and define $b$ as its image in $H^{0}(\Tw \B)(X,U)$.
        We also define
        $$q:= (0, v) \in \Tw \B (C(f),U)^{0}.$$
        This latter map will be the homotopy that makes the rightmost square in the diagram below commute:
        \begin{equation} \label{eq - final diagram} \begin{tikzcd}
	X & Y & Z & U \\
	X & Y & {C(f)} & {\Sigma X.}
	\arrow["f", from=1-1, to=1-2]
	\arrow[equals, from=1-1, to=2-1]
	\arrow["g", from=1-2, to=1-3]
	\arrow["h", from=1-3, to=1-4]
	\arrow["f"', from=2-1, to=2-2]
	\arrow[equals, from=2-2, to=1-2]
	\arrow["i"', from=2-2, to=2-3]
	\arrow["a", from=2-3, to=1-3]
	\arrow["q", dashed, from=2-3, to=1-4]
	\arrow["p"', from=2-3, to=2-4]
	\arrow["b"', from=2-4, to=1-4]
\end{tikzcd}\end{equation}

 We already showed in \zcref{lem - def of btilde} that $\tilde{b}$ is closed. There remains to be shown that \zcref{eq - final diagram} commutes. On the one hand, we find
        \begin{align*}
            \quad b_{\Tw \B,2}(s\tilde{h},s\tilde{a}) 
            &= b_{2}(s\tilde{h},s\tilde{a}) + b_{3}(s\tilde{h},s\tilde{a},s\delta) \\
            &= b_{2}(s\tilde{h},s(-\sigma^{-1}u \quad \tilde{g}))) + b_{3}(s\tilde{h},s(-\sigma^{-1}u \quad \tilde{g}),s\begin{pmatrix}
                0 & 0 \\ \sigma^{-1}\tilde{f} & 0
            \end{pmatrix}) \\
            &= (b_{2}(s\tilde{h},\sigma^{-1}(su)) \quad b_{2}(s\tilde{h},s\tilde{g})) + (b_{3}(s\tilde{h}, s\tilde{g}, s\sigma^{-1}\tilde{f}) \quad 0) \\ 
            &= (-\sigma^{-1} b_{2}(s\tilde{h},su)- \sigma^{-1} b_{3}(s\tilde{h},s\tilde{g},s\tilde{f})) \quad b_{1}(sv)).
        \end{align*}
On the other hand, we have
        \begin{align*}
            b_{\Tw \B,2}&(s\tilde{b},sp) 
            = b_{2}(s\tilde{b},sp) + b_{3}(s\tilde{b},sp,s\delta) \\
	&=b_2(s\tilde{b},s(\sigma^{-1}1 \quad 0))+ b_{3}(s\tilde{b},s\begin{pmatrix}
                \sigma^{-1}1 & 0
            \end{pmatrix},s \begin{pmatrix}
                0 & 0 \\ \sigma^{-1}\tilde{f} & 0
            \end{pmatrix}) \\
	&=(-b_2(s\tilde{b},\sigma^{-1}s1) \quad 0) \\
	&=(-\sigma^{-1}b_2(s\tilde{b},s1) \quad 0) \\
	&= (-\sigma^{-1}s\tilde{b} \quad 0)\\
	&=(\sigma^{-1}b_{2}(s\tilde{h},su) - \sigma^{-1} b_{2}(sv,s\tilde{f})  +\sigma^{-1}b_{3}(s\tilde{h},s\tilde{g},s\tilde{f}) \quad 0)
        \end{align*}
        since $b_{2}(sx,s1) = (-1)^{\vert x \vert} sx$ (as before, the second term vanishes because $b_{\tn{Free}(\B)^{\oplus 2},3}$ decomposes into expressions of $b_{\B,3}$ in which every term has a zero component).

And finally, 
\begin{align*}
            b_{\Tw \B,1}(sq) &= b_{1}(sq) + b_{2}(sq,s\delta) \\
            &=b_{1}(s(0 \quad v)) + b_{2}(s( 0 \quad v), s\begin{pmatrix}
                0 & 0 \\ \sigma^{-1}\tilde{f} & 0
            \end{pmatrix}) \\
            &=(0 \quad b_{1}(sv)) + (b_{2}(sv,s\sigma^{-1}\tilde{f}) \quad 0) \\
            &= (b_{2}(sv,s\sigma^{-1}\tilde{f}) \quad b_{1}(sv)) \\
            &= (-\sigma^{-1}b_{2}(sv,s\tilde{f}) \quad b_{1}(sv)).
\end{align*}

Putting this all together, we have that 
\begin{align*}
b_{\Tw \B,2}(s\tilde{b},sp) +  b_{\Tw \B,2}(s\tilde{h},s\tilde{a}) &=
(-\sigma^{-1}b_2(sv,s\tilde{f}) \quad b_1(sv) \\
&=b_{\Tw \B,1}(sq).
\end{align*}
This yields commutativity of the rightmost square: Indeed, the two summands are taken with a plus sign so that they commute on the $m$-level using \zcref{eq - b2 to m2} (this is because $\tilde{h}$ has degree 1). \qedhere
        \end{enumerate}
    \end{proof}
\end{lemma}

Together with \zcref{prop - Classic Massey}, we can conclude from this that the Massey products of \zcref{def - Massey Ainf} allow us to detect the distinguished triangles in $H^{0}(\B)$ (without any reference to $H^{0}(\Tw \B)$). This proves \zcref{prop - Massey Ainfty}.

% ----------------------------------------------------------------------------------------
% BIBLIOGRAPHY

\medskip
\printbibliography
\ver
\ver

\end{document}